\documentclass[leqno]{article}

\usepackage{etex}
\usepackage[papersize={210mm,297mm},margin=0.75in ]{geometry}
\usepackage{amsmath, bm,amssymb,amsfonts,dsfont,amsthm}
\usepackage{graphicx}
\usepackage[usenames,dvipsnames,svgnames,table]{xcolor}
\usepackage{multicol}
\usepackage{mathrsfs}
\usepackage{mathtools}
\usepackage{amsthm}
\usepackage{enumitem}
\usepackage{lineno}
\definecolor{st.patricksblue}{rgb}{0.14, 0.16, 0.48}
	\definecolor{indigo(web)}{rgb}{0.29, 0.0, 0.51}
\definecolor{sacramentostategreen}{rgb}{0.0, 0.34, 0.25}
\usepackage[colorlinks=true,linkcolor=indigo(web), urlcolor=st.patricksblue, filecolor=blue, citecolor=sacramentostategreen,backref=page]{hyperref}

\usepackage{etoolbox}
\makeatletter
\patchcmd{\BR@backref}{\newblock}{\newblock(Cited on page~}{}{}
\patchcmd{\BR@backref}{\par}{)\par}{}{}
\makeatother

\usepackage{relsize}

\usepackage[all]{xy}
\usepackage{tikz-cd}
\usepackage{array}
\usepackage{tensor}
\usepackage[cal=mathpi,calscaled=.94, bb=ams,frakscaled=.97,scr=rsfso]{mathalfa} 
\usepackage[utf8]{inputenc}

\usetikzlibrary{matrix}

\newcommand{\ds}[1]{\ensuremath{\mathds{#1}}}
\newcommand{\upa}[1]{\ensuremath{\upsilon^{{#1}/p^\infty}}}

\newcommand{\klr}{\ensuremath{K\lr{\upsilon^{\pm1/p^\infty}}}}
\newcommand{\kl}{\ensuremath{K\lr{\upsilon^{1/p^\infty}}}}

\newcommand{\curly}[1]{\ensuremath{\mathscr{#1}}}

\newcommand{\pro}{\ensuremath{\ds{P}_K^{1,\mathrm{ad,perf}}}}

\newcommand{\aro}{\ensuremath{\ds{A}_K^{1,\mathrm{ad,perf}}}}
\newcommand{\arom}{\ensuremath{\ds{A}_K^{m,\mathrm{ad,perf}}}}

\newcommand{\tio}[1]{\ensuremath{{#1}^\times}}

\newcommand{\lr}[1]{\left\langle{#1}\right\rangle}

\newcommand{\id}[1]{\ensuremath{{\mathfrak{#1}}}}

\newcommand{\abs}[1]{\ensuremath{{\left\vert{#1}\right\vert}}}

\newcommand{\bs}{\ensuremath{\backslash}}
\newcommand{\xra}[1]{\xrightarrow{#1}}
\newcommand{\ra}{\rightarrow}
\newcommand{\til}[1]{\widetilde{#1}}

\usepackage{ccfonts,eucal}
\usepackage[euler-digits,euler-hat-accent]{eulervm}
\usepackage{sseq}

\newtheorem{theorem}{Theorem}[section]

\newtheorem{proposition}[theorem]{Proposition}
\newtheorem{corollary}[theorem]{Corollary}
\newtheorem{mdef}[theorem]{Definition}
\let\olddefinition\mdef
\renewcommand{\mdef}{\olddefinition\normalfont}

\newtheorem{exam}[theorem]{Example}
\let\oldexample\exam
\renewcommand{\exam}{\oldexample\normalfont}

\newtheorem{rem}{Remark}
\let\oldremark\rem
\renewcommand{\rem}{\oldremark\normalfont}

\begin{document}

\title{Vector Bundles over Projectivoid Line}
\author{Harpreet Singh Bedi~~~{bediha@dickinson.edu}}

\maketitle

\begin{abstract}
In this article we describe vector bundles over projectivoid line and show how it is similar to (and different) from Gorthendieck's classification of vector bundles over projective line.
\end{abstract}

\section*{Acknowledgement} I am very grateful to Kiran Kedlaya for suggesting this topic and for critical comments. I am responsible for any remaining errors in the article.

\tableofcontents

\section{Introduction}
In this article we describe vector bundles over projectivoid line $\pro$ in Propostion \ref{prop:f1}. The description will be similar to vector bundles on $\ds{P}^1$ as described in \cite{martin2} Proposition 2.3. We reproduce the proposition here (word for word),

\begin{proposition}\label{prop:martin_2_3}
\normalfont
Isomorphism classes of $m$ dimensional algebraic vector bundles over $\ds{P}^1$ correspond bijectively to equivalence classes of polynomial $m\times m$ matrices $A(s,s^{-1})$ over $k[s,s^{-1}]$ such that $\det A(s,s^{-1})=s^n,n\in\ds{Z}$ where equivalence relation is the following: $A(s,s^{-1})\sim A'(s,s^{-1})$ iff there exist polynomial invertible $m\times m$ matrices $U(s),V(s^{-1})$ over $k[s]$ and $k[s^{-1}]$ respectively with constant determinant such that

\begin{equation}
A'(s,s^{-1})=V(s^{-1})A(s,s^{-1})U(s).
\end{equation}
\end{proposition}

\section{Notation}
We start by defining the following \klr $K$-Algebras (perfectoid algebras). Let $a,b\in\ds{N}\cup\{0\}$ and $i=a/p^b$,then we define

\begin{align}
K\lr{\upsilon^{1/p^\infty}}&:=\sum_{a,b}c_{(a,b)}\upsilon^{i}, ~~c_{(a,b)}\in K,~ \lim_{a+b\ra\infty}|c_{(a,b)}|\ra0\\
K\lr{\upsilon^{-1/p^\infty}}&:=\sum_{a,b}c_{(a,b)}\frac{1}{\upsilon^{i}}, ~~ c_{(a,b)}\in K,~ \lim_{a+b\ra\infty}|c_{(a,b)}|\ra0\\
K\lr{\upsilon^{\pm1/p^\infty}}&:=\text{ Generated by }\alpha,\beta \text{ where } \alpha
\in K\lr{\upsilon^{1/p^\infty}}\text{ and }\beta\in K\lr{\upsilon^{-1/p^\infty}}
\end{align}
For the ease of notation we will write $c_i$ (or even $a_i$) in place of $c_{(a,b)}$. 

It is possible to put an order on the objects defined above. The order is not important but we give one such order below. Another such order can be obtained using the Calkin-Wilf Sequence as given in \cite{Cal_Wi}.

\subsection{Order and Grading}\label{sec:grading} 
Polynomials come equipped with standard grading, but here we are working with power series with degree of individual terms of the form $a/p^b\in\ds{Q}$ where $a,b\in\ds{Z}$ and $p$ a prime. We have to fix a convention for expressing terms as summation, and we make sure that there are finitely many terms in each grading. First we grade $K[\upsilon,\upsilon^{1/p},\upsilon^{1/p^2},\ldots]$. 

Consider antidiagonal in the first quadrant, it consists of terms $(a,b)$ with $a,b\in\ds{N}\cup\{0\}$. The sum of the terms is fixed say $k\in\ds{N}\cup\{0\}$. For example, corresponding to $k=3$ we have the following tuples $(0,3),(1,2),(2,1),(3,0)$ as $(a,b)$, and every antidiagonal has a fixed number of terms in the first quadrant. We will use this as a model for grading. The term $(a,b)$ will correspond to $\upsilon^{a/p^b}$. The terms on the $x$ axis of the form $(a,0)$ give us the grading on the polynomial in $\upsilon$, and as we go to higher and higher antidiagonal we keep recovering higher powers of $1/p$. The vertical line $x=1$  gives us just the powers of $\upsilon$ in $1/p$. As the reader would have noticed, the notation follows the proof of countability of rationals, skipping any duplicate terms.

Our polynomials are finite sums of the form
\begin{equation}
\sum_{a+b=i}a_i\upsilon^{a/p^b}, ~~a,b,i\in\ds{N}\cup\{0\}, a_i\in K
\end{equation} 
 (there is no relation between $a_i$ and $a$) and can be clearly extended to power series by making the sum infinite, we denote power series by $K\lr{\upsilon}$.  In case of power series we also add an extra condition that $|{a_i}|\ra 0$ as $i\ra \infty$. Laurent polynomials can be added by duplicating the above sum (we will still have finitely many terms in the antidiagonal and thus grading).

\section{Vector Bundles over \pro} The projectivoid line is covered by $\id{U}=U_1\cup U_2$, with $\curly{O}(U_1)=K\lr{\upsilon^{1/p^\infty}}$  and $\curly{O}(U_2)=K\lr{\upsilon^{-1/p^\infty}}$. Where $U_i$ is the perfectoid affine space \aro and $\curly{O}(U_1\cap U_2)=K\lr{\upsilon^{\pm1/p^\infty}}$. The sheafiness comes from Theorem 3.6.15 of \cite{Kiran_Liu_1}.

Let $E$ be a $m$-bundle over \pro defined over $K$.

There are two trivilizations of this bundle over the cover corresponding to $U_1$ and $U_2$. The trivialization is of the form $U_i\times \arom=\aro\times \arom$. Let $s\in\aro$ and $v\in\arom$. To construct the projectivoid we identify the perfectoid affine spaces via the map $s\mapsto 1/s,s\neq 0$. Now we can glue the two trivializations of the vector bundle to get a vector bundle over the projectivoid space. 

\begin{align*}
U_1\bs\{0\}\times\arom &\ra U_2\bs\{0\}\times\arom\\
(s,v)&\mapsto (s^{-1},A(s,s^{-1})v)\\
\end{align*}
where $A(s,s^{-1})$ is a matrix with coefficients in $\curly{O}(U_1\cap U_2)=K\lr{\upsilon^{\pm1/p^{\infty}}}$. For the correspondence to hold this matrix must have a determinant that is a unit in the ring. The determinant is a power series.

\begin{equation}\label{eq:trans_matrix}
\det (A(s,s^{-1}))\in K\lr{\upsilon^{\pm1/p^{\infty}}}\text{ and }\det (A(s,s^{-1}))\neq 0\text{ for all     }~~\upsilon
\end{equation}
As we see in the next section \ref{sec:Poly_and_power_series} the units of \eqref{eq:trans_matrix}  are given as 

\begin{equation}\label{eq:det_trans_matrix_1}
{\upsilon}^{n/p^b}\cdot v ,~~~n\in\ds{Z}, b\in \ds{Z}_{>0},
\end{equation}
where $v$ is degree zero term of $\klr$.

Notice that if we restrict to the case of $k[s,s^{-1}]$ we end up getting determinant as $s^n$, as in the proposition \ref{prop:martin_2_3}

\section{Polynomials and Power Series} \label{sec:Poly_and_power_series}The units in the ring $K[X]$ are precisely $\tio{K}$, and for the laurent polynomials $K[X,X^{-1}]$ the units are $uX^n,u\in \tio{K}$.

In the case of power series $K[[X]]$ the units are formal power series with non zero constant term.

\begin{equation}
\sum_{n=0}^\infty a_nx^n\in K[[X]]\text{ is a unit iff  }a_0\neq 0.
\end{equation}
In the case of formal Laurent series $K((X))$, we notice that $X$ is a unit, since $X^{-1}\cdot X=1$. The set of units is ${K((X))}\bs{0}$, the proof can be seen in \cite{mo1}. In \cite{SCHWAIGER} we find the complete description of roots of power series.

For a series $f$ in Tate Algebra $T_n$, the series is a unit iff the constant coefficient of $f$ is bigger than all other coefficients of $f$ ( \cite{bosch2014lectures} page 14, Corollary 4). For $T_n:=K\lr{\upsilon_1,\ldots,\upsilon_n}$ the units are 

\begin{equation}\label{eq:tate_unit}
\tio{T_n}=\left\{\sum_{i\in\ds{N}^n} a_i\upsilon^i\in T_n ~:~|a_0|> |a_i|\text{ for all }i\neq0     \right\}
\end{equation}
Equipping $f\in T_n$ with a Gauss norm, $|f|=1$ is a unit iff  the reduction of $f$ denoted as $\til{f}$ lies in $ \tio{K}$ as described on  ( \cite{bosch2014lectures} page 13, page 14, Corollary 4).
\subsection{Units of $K\lr{\upsilon^{1/p^\infty}}$}
We now formally write down the units of $K\lr{\upsilon^{1/p^\infty}}$ which will be used for the description of vector bundles.
\begin{proposition}
 $K\lr{\upsilon^{1/p^\infty}}$ is complete with respect to Gauss Norm.
\end{proposition}

\begin{proof} 
This proof is an adaptation of a similar proposition for Tate Algebras as given in  ( \cite{bosch2014lectures} page 14, Proposition 3).  We start with a Cauchy sequence $\sum_i f_i$ and end up showing that it lies in $ \kl$. We will use $v$ as an index for $(a,b)$, this will help as streamline the proof to make it closer to the proof of units of Tate Algebra.
\begin{equation}
\lim_{i\ra\infty}f_i =0 \text{ where }f_i=\sum_vc_{iv}\upsilon^v\in \kl
\end{equation}
First note that 
\begin{equation}
\abs{c_{iv}}\leq\abs{f_i}\text{ thus }\lim_{i\ra\infty}\abs{c_{iv}}=0 \text{ for all }v.
\end{equation}
Thus, the limit $c_v=\sum_{i=0}^\infty c_{iv}$ exists (note that we are using Gau\ss ~norm). 
To finish the proof we need to show that the series $f=\sum_vc_v\upsilon^v$ is strictly convergent and $f=\sum_if_i$.

In the section \ref{sec:grading} we put an order on $\kl$. In order to make our argument simpler, we jump a finite number of terms (we noticed in our ordering that there are only finite number of terms for every grading) in the order given in \ref{sec:grading} and consider terms of the form $(a,0)$ lying on the $x$ axis. This helps us in thinking directly in terms of natural numbers $\ds{N}$.

For any given $\epsilon>0$ there is an integer $N$ such that $\abs{c_{iv}}<\epsilon$ for $i\geq N$ and all $v$. Since coefficients  of the series $f_0,\ldots, f_{N-1}$ form a zero sequence, and almost all the coefficients of these sequences would have an absolute value less than $\epsilon$. Thus, the elements $\abs{c_{iv}}$ form a zero sequence in $K$. 
Since the non Archimedean triangle inequality generalizes for a convergent series to an inequality below \begin{equation}
\abs{\sum_{i=0}^\infty\alpha_v}\leq \max_{i=0,\ldots,\infty}\abs{\alpha_v}
\end{equation}
we get that power series $f=\sum_if_i$ and $f\in \kl$.
\end{proof}

\begin{corollary}
 A series $f\in \kl$ with $\abs{f}=1$ is a unit iff its reduction $\til{f}\in\tio{k}$.
\end{corollary}
\begin{proof}
Without loss of generality we can consider only elements with $f\in \kl$ with Gau\ss ~norm 1. If $f$ is a unit in \kl it is also a unit in $R\lr{\upa{1}}$, where 

\begin{align}
R&=\{a\in K~|~ \abs{a}\leq 1\}\\
\id{m}&=\{a\in K~|~ \abs{a}< 1\}\\
k&=R/\id{m}\\
R\lr{\upa{1}}&\ra k[\upa{1}]\\
f&\mapsto \til{f}
\end{align}
Thus, $\til{f}$ is a unit in $k[\upa{1}]$ and hence in $f\in\tio{k}$. 

Conversely, if $\til{f}\in\tio{k}$, the constant term $f(0)$ satisfies $\abs{f(0)}=1$ (since $\til{f}=0$ iff $\abs{f}<1$). But then we can put $f=1-g$ with $\abs{g}<1$, giving us an inverse of $f$ as a series $\sum_{i=0}^\infty g^i$.

\end{proof}

In the above corollary we showed $f$ is of the type $f=1-g$ with $\abs{g}<1$. Thus, we can restate the above corollary as 

\begin{corollary}
An arbitrary series $f\in\kl$ is a unit iff $\abs{f-f(0)}<\abs{f(0)}$. In other words the absolute value of other coefficients of $f$ are less than the absolute value of the constant coefficient.
\begin{equation}\label{eq:1p_tate_unit}
\tio{\kl}=\left\{\sum_{i\in\ds{N}^n} a_i\upsilon^i\in \kl ~:~|a_0|> |a_i|\text{ for all }i\neq0     \right\}
\end{equation}
\end{corollary}

We can carry the exact same procedure as above for $K\lr{\upsilon^{-1/p^\infty}}$ to get an identical result as stated in \ref{eq:1p_tate_unit}.

\subsection{Units of \klr}

We can consider algebra of the form $K\lr{X^{1/p^\infty},Y^{1/p^\infty}}$. An element $f\in K\lr{X^{1/p^\infty},Y^{1/p^\infty}}$ is a series in which each individual term has a degree $=\deg X+\deg Y$, where $X$ and $Y$ occur in the term. Thus, we can put an order on these terms as given in section \ref{sec:grading}. If we have terms which have only $X$ or only $Y$ appearing in them, we can still arrange them by degree. In case the degree of $X$ and $Y$ term is same, we put an order by first writing the $X$ term and then the $Y$ terms of the same degree. The order simply comes from observing that rational numbers are countable.

Using the results (and procedure) from the previous section we get the units of $K\lr{X^{1/p^\infty},Y^{1/p^\infty}}$ given below where $\xi$ represents product of $X$ and $Y$.
\begin{equation}\label{eq:2p_tate_unit}
\left\{\sum_{i\in\ds{N}^n} a_i\xi^i\in K\lr{X^{1/p^\infty},Y^{1/p^\infty}} ~:~|a_0|> |a_i|\text{ for all }i\neq0     \right\}
\end{equation}

Setting $X=\upsilon$ and $Y=1/\upsilon$  we know that elements of the form below 

\begin{equation}\label{eq:p_tate_unit}
\upsilon^{n/p^b}\left\{\sum_{i\in\ds{N}^n} a_i\upsilon^i\in \klr ~:~|a_0|> |a_i|\text{ for all }i\neq0     \right\}
\end{equation}

are units in $\klr$. The units we are most interested in are of the form $\upsilon^{n/p^b}\cdot u$ where $u$ is a degree zero term of \klr. We define the notion of degree in the definition given at \ref{def:degree} . 

It might seem that there are other units of $\klr$ that might not have a clearly defined notion of degree. 

For other units of $\klr$ we notice that the tail ends of series on both positive and negative side tend to zero. Thus, there are only finitely many terms that could be dominant. We can still define the degree to be maximum degree of all dominant terms (which are finite in number). In case we just have a polynomial with all the coefficients equal, then we have the degree is the power of the highest term, which is same as degree of polynomial in the classical case.

Hence we have a well defined notion of degree for units which might not be of the form \eqref{eq:p_tate_unit}. We can take the degree term out and write the unit as $\upsilon^{n/p^b}\cdot u'$, where the degree of $u'$ is zero.

\section{Isomorphism Classes of Vector Bundles over \aro} We now want to talk about vector bundle automorphism over the space $\aro\times\arom$. 

\begin{align*}
U_1\times\arom &\ra U_1\times\arom\\
(s,v)&\mapsto (s,U(s)v)\\
\end{align*}
where $U(s$) is a matrix with coefficients in $K\lr{\upsilon^{1/p^\infty}}$ and $\det(U(s))\neq 0$. From the section  \ref{sec:Poly_and_power_series} we get the units as
\begin{equation}\label{eq:det_trans_matrix1}
\left\{\text{ Elements of }K\lr{\upsilon^{1/p^{\infty}}}\text{ such that }~|a_0|> |a_i|\text{ for all }i\neq0     \right\}
\end{equation}
Notice that there is no gluing condition just on the piece $U_1$, therefore can have $\upsilon=0$. If we restrict this to $k[s]$ we just get $k\bs\{0\}$ as in the proposition \ref{prop:martin_2_3}.

 Similarly we have a correspondence on $U_2$

\begin{align*}
U_2\times\arom &\ra U_2\times\arom\\
(t,v)&\mapsto (t,V(t)v)\\
\end{align*}
where $V(t)$ is a matrix with coefficients in $K\lr{\upsilon^{1/p^\infty}}$ and $\det(V(t))\neq 0$. Notice that to obtain the projectivoid space we will have $t=1/s$. Thus, we write $V(s^{-1})$ in place of $V(t)$.

From the section \ref{sec:Poly_and_power_series} we get the units as
\begin{equation}\label{eq:det_trans_matrix2}
\left\{\text{ Elements of }K\lr{\upsilon^{-1/p^{\infty}}}\text{ such that }~|a_0|> |a_i|\text{ for all }i\neq0     \right\}
\end{equation}
Notice that there is no gluing condition just on the piece $U_2$, therefore can have $\upsilon=0$. If we restrict this to $k[t]$ we just get $k\bs\{0\}$ as in the proposition \ref{prop:martin_2_3}.

We want an equivalence relation for transition matrix between two covers. This can be obtained modulo the automorphisms $U(s),V(s^{-1})$ and is given as \eqref{eq:equivrel}.

\begin{align}
U_1\times\arom\xra{U(s)}&U_1\times\arom\xra{A(s,s^{-1})}U_2\times\arom\xra{V(s^{-1})}U_2\times\arom\nonumber\\
&U_1\times\arom\xra{A'(s,s^{-1})}U_2\times\arom\nonumber\\
&A'(s,s^{-1})=V(s^{-1})A(s,s^{-1})U(s)\label{eq:equivrel}
\end{align}

\section{Degree of Vector Bundles}

In this section we define the notion of degree of the vector bundles on the projectivoid line, which is motivated by proposition \ref{prop:martin_2_3}. In this proposition we have

\begin{align}
&\deg \det U(s)=0=\deg\det  V(s^{-1})\text{ which implies }\\
& \deg\det  A'(s,s^{-1})=\deg\det  A(s,s^{-1})
\end{align}
Thus, isomorphic vector bundles on $\ds{P}^1$ have the same degree of the determinant

\begin{mdef}\label{def:degree}
Degree of the vector bundle is the degree of zeroth term of the determinant of the vector bundle. 
\end{mdef}

The consequence of the above definition is that \eqref{eq:tate_unit} will have degree zero. Thus, \eqref{eq:det_trans_matrix1} and \eqref{eq:det_trans_matrix2} will also have degree zero. Furthermore the degree of \eqref{eq:det_trans_matrix_1} is $n/p^{\infty}$. From, \eqref{eq:equivrel} and the observations just made 
\begin{equation}
 \deg\det  A'(s,s^{-1})=\deg\det  A(s,s^{-1})
\end{equation}
Thus, isomorphic vector bundles on $\pro$ have the same degree of the determinant.

We have proved the following proposition

\begin{proposition}\label{prop:f1}
Isomorphism classes of $m$ dimensional analytic vector bundles over $\pro$ correspond bijectively to equivalence classes of $m\times m$ matrices $A(s,s^{-1})$ over $K\lr{\upsilon^{\pm1/p^{\infty}}}$ such that  
\begin{equation}
\det A(s,s^{-1})={\upsilon}^{n/p^b}\cdot\left\{\text{ Elements of }K\lr{\upsilon^{\pm1/p^{\infty}}}\text{ such that }~|a_0|> |a_i|\text{ for all }i\neq0     \right\}
\end{equation}

where equivalence relation is the following: $A(s,s^{-1})\sim A'(s,s^{-1})$ iff there exist invertible $m\times m$ matrices $U(s),V(s^{-1})$ over $K\lr{\upsilon^{1/p^\infty}}$ and $K\lr{\upsilon^{-1/p^\infty}}$ respectively with determinants of $U(s)$ and $V(s)$ given as 

\begin{align}
\det U(s)&=\left\{\text{ Elements of }K\lr{\upsilon^{1/p^{\infty}}}\text{ such that }~|a_0|> |a_i|\text{ for all }i\neq0     \right\}\\
\det V(s^{-1})&=\left\{\text{ Elements of }K\lr{\upsilon^{-1/p^{\infty}}}\text{ such that }~|a_0|> |a_i|\text{ for all }i\neq0     \right\}
\end{align}
such that

\begin{equation}
A'(s,s^{-1})=V(s^{-1})A(s,s^{-1})U(s).
\end{equation}
and 

\begin{equation}
 \deg\det  A'(s,s^{-1})=\deg\det  A(s,s^{-1})
\end{equation}
\end{proposition}

\section{Classification of Vector Bundles on $\pro$}
The classification of vector bundles over $\ds{P}^1$ depends upon the fact that there are only finitely many ways to partition an integer. But, this is no longer true for fractions. For example, consider the following non equivalent (and infinitely many) vector bundles with degree one.

\begin{equation}
 \begin{pmatrix}
 X^a & 0\\
  0 & X^b 
  \end{pmatrix}
\text{ such that }a+b=1 \text{ and } a,b\in[0,1]\cap\ds{Z}[1/p]
\end{equation}

A more subtle question is whether every vector bundle on $\pro$ splits as a sum of line bundles. This question was answered in the negative by Prof Kiran Kedlaya and communicated to me via email. The counterexample is mentioned in full detail in Lecture 3 of \cite[pp.80-81]{Kiran_AWS}.

\bibliographystyle{apalike}
\bibliography{qhe}
\end{document}